\documentclass[reqno,12pt]{amsart}
\usepackage{amscd,amsfonts,amssymb}
\usepackage[T2A]{fontenc}
\numberwithin{equation}{section}
\newtheorem{Theorem}{Theorem}[section]
\newtheorem{Lemma}{Lemma}[section]
\newtheorem{Corollary}{Corollary}[section]

\theoremstyle{definition}

\theoremstyle{remark}
\newtheorem{Remark}{Remark}[section]
\newtheorem{Example}{Example}[section]

\author{A.\,A. Kon'kov}
\address{Department of Differential Equations,
Faculty of Mechanics and Mathematics,
Mo\-s\-cow Lo\-mo\-no\-sov State University,
Vorobyovy Gory,
Moscow, 119992 Russia.
}
\email{konkov@mech.math.msu.su}
\author{A.\,E. Shishkov}
\address{
Center of Nonlinear Problems of Mathematical Physics,
RUDN University,
Mi\-klu\-k\-ho-Mak\-la\-ya str. 6,
Moscow, 117198 Russia.
}
\email{aeshkv@yahoo.com}
\title[On bow-up conditions]{On bow-up conditions for systems of higher order differential inequalities}
\date{}

\begin{document}

\begin{abstract}
We consider systems of the differential inequalities
$$
	\left\{
		\begin{aligned}
			&
			\sum_{|\alpha| = m_1}
			\partial^\alpha
			a_\alpha (x, u_1)
			\ge
			f_1 (u_2)
			&
			\mbox{in } {\mathbb R}^n,
			\\
			&
			\sum_{|\alpha| = m_2}
			\partial^\alpha
			b_\alpha (x, u_2)
			\ge
			f_2 (u_1)
			&
			\mbox{in } {\mathbb R}^n,
		\end{aligned}
	\right.
$$
where $n, m_1, m_2 \ge 1$ are integers and $a_\alpha$ and $b_\alpha$ are Caratheodory functions such that
$$
	|a_\alpha (x, \zeta)| + |b_\alpha (x, \zeta)| \le A |\zeta|
$$
with some constant $A > 0$ for almost all $x \in {\mathbb R}^n$ and for all $\zeta \in {\mathbb R}$.
For solutions of these systems exact blow-up conditions are obtained.
\end{abstract}

\maketitle

\section{Introduction}

We consider systems of the differential inequalities
\begin{equation}
	\left\{
		\begin{aligned}
			&
			\sum_{|\alpha| = m_1}
			\partial^\alpha
			a_\alpha (x, u_1)
			\ge
			f_1 (u_2)
			&
			\mbox{in } \mathbb R^n,
			\\
			&
			\sum_{|\alpha| = m_2}
			\partial^\alpha
			b_\alpha (x, u_2)
			\ge
			f_2 (u_1)
			&
			\mbox{in } \mathbb R^n,
		\end{aligned}
	\right.
	\label{1.1}
\end{equation}
where $n, m_1, m_2 \ge 1$ are integers and $a_\alpha$ and $b_\alpha$ are Caratheodory functions such that
$$
	|a_\alpha (x, \zeta)| + |b_\alpha (x, \zeta)| \le A |\zeta|
$$
with some constant $A > 0$ for almost all $x \in {\mathbb R}^n$ and for all $\zeta \in {\mathbb R}$.
In so doing, the functions $f_1, f_2 : [0, \infty) \to [0, \infty)$ on the right in~\eqref{1.1} are assumed to be non-decreasing and convex on the interval $[0, \infty)$.
As is customary, by $\alpha = {(\alpha_1, \ldots, \alpha_n)}$ we mean a multi-index with
$|\alpha| = \alpha_1 + \ldots + \alpha_n$ 
and
$
	\partial^\alpha 
	= 
	{\partial^{|\alpha|} / (\partial_{x_1}^{\alpha_1} \ldots \partial_{x_n}^{\alpha_n})}.
$
Let us also denote by $B_r$ the open ball of radius $r > 0$ centered at zero.

We say that $u_1, u_2 \in L_{1, loc} (\mathbb R^n)$ are solutions of~\eqref{1.1} if $f_1 (u_2), f_2 (u_1) \in L_{1, loc} (\mathbb R^n)$ and, moreover,
\begin{equation}
	(-1)^{m_1}
	\int_{\mathbb R^n}
	\sum_{|\alpha| = m_1}
	a_\alpha (x, u_1)
	\partial^\alpha
	\varphi
	\,
	dx
	\ge
	\int_{\mathbb R^n}
	f_1 (u_2)
	\varphi
	\,
	dx
	\label{1.2}
\end{equation}
and
\begin{equation}
	(-1)^{m_2}
	\int_{\mathbb R^n}
	\sum_{|\alpha| = m_2}
	b_\alpha (x, u_2)
	\partial^\alpha
	\varphi
	\,
	dx
	\ge
	\int_{\mathbb R^n}
	f_2 (u_1)
	\varphi
	\,
	dx
	\label{1.3}
\end{equation}
for any non-negative function $\varphi \in C_0^\infty (\mathbb R^n)$.

The absence of non-trivial solutions of differential inequalities, which is also known as the blow-up phenomenon, has traditionally attracted the attention of many mathematicians~[1--12].
However, most studies were limited to the case of power-law nonlinearities or dealt with second-order differential operators.

In our paper, an attempt is made to study nonlinearities other than power-law ones, while differential operators of arbitrary order are considered.
We obtain exact conditions guaranteeing that all solutions of~\eqref{1.1} are trivial, i.e. these solutions are equal to zero almost everywhere in $\mathbb R^n$.
For power-law nonlinearities, our results imply blow-up conditions obtained in~\cite[Theorem~19.1]{MPbook}.

Below it is assumed
that there be real numbers $\lambda_i, \mu_i \in [1, \infty)$ and non-decreasing continuous functions $g_i : [1, \infty) \to (0, \infty)$ and $h_i : (0, 1] \to (0, \infty)$ satisfying the conditions
\begin{equation}
	f_i (\zeta) \ge \zeta^{\lambda_i} g_i (\zeta)
	\quad
	\mbox{for all } \zeta \in [1, \infty),
	\quad 
	i = 1,2,
	\label{1.4}
\end{equation}
and
\begin{equation}
	f_i (\zeta) \ge \zeta^{\mu_i} h_i (\zeta)
	\quad
	\mbox{for all } \zeta \in (0, 1],
	\quad 
	i = 1,2,
	\label{1.5}
\end{equation}
and, moreover, 
\begin{equation}
	\operatorname*{lim\,inf}\limits_{\zeta \to \infty}
	\frac{
		g_i (\zeta^\varkappa)
	}{
		g_i (\zeta)
	}
	>
	0
	\label{1.6}
\end{equation}
and
\begin{equation}
	\operatorname*{lim\,inf}\limits_{\zeta \to +0}
	\frac{
		h_i (\zeta)
	}{
		h_i (\zeta^\varkappa)
	}
	>
	0
	\label{1.7}
\end{equation}
for all $\varkappa \in (0, 1)$, $i = 1,2$.

\section{Main results}

\begin{Theorem}\label{T2.1}
Suppose that $\mu_1 \mu_2 > 1$,
\begin{equation}
	\int_1^\infty
	(
		\zeta^{\lambda_1 \lambda_2}
		g_1 (\zeta)
		g_2^{\lambda_1} (\zeta)
	)^{
		- 1 / (m_1 + \lambda_1 m_2)
	}
	\zeta^{
		1 / (m_1 + \lambda_1 m_2) - 1
	}
	d\zeta
	<
	\infty,
	\label{T2.1.1}
\end{equation}
and
\begin{equation}
	\int_0^1
	r^{(n - m_2) \mu_1 \mu_2 - m_1 \mu_2 - n - 1}
	h_1^{\mu_2} (r)
	h_2 (r)
	\,
	dr
	=
	\infty.
	\label{T2.1.2}
\end{equation}
Then all solutions of~\eqref{1.1} are trivial.
\end{Theorem}

\begin{Corollary}\label{C2.1}
Let $\lambda_1 \lambda_2 > 1$ and, moreover,~\eqref{T2.1.2} be satisfied.
Then all solutions of~\eqref{1.1} are trivial.
\end{Corollary}

In the case of $m_1 \ge n$, condition~\eqref{T2.1.2} in Theorem~\ref{T2.1} and Corollary~\ref{C2.1} can be dropped. Namely, the following statements are valid.

\begin{Theorem}\label{T2.2}
Suppose that $m_1 \ge n$ and~\eqref{T2.1.1} holds.
Then all solutions of~\eqref{1.1} are trivial.
\end{Theorem}

\begin{Corollary}\label{C2.2}
Let $m_1 \ge n$ and $\lambda_1 \lambda_2 > 1$.
Then all solutions of~\eqref{1.1} are trivial.
\end{Corollary}

\begin{Theorem}\label{T2.3}
Suppose that $\mu_1 = \mu_2 = 1$ and~\eqref{T2.1.1} holds. 
Then all solutions of~\eqref{1.1} are trivial.
\end{Theorem}

\begin{Remark}\label{R2.1}
Nothing obviously prevents us from interchanging the inequalities and unknown functions in~\eqref{1.1}. Thus, Theorems~\ref{T2.1} and~\ref{T2.3} and Corollary~\ref{C2.1} remain valid with conditions~\eqref{T2.1.1} and~\eqref{T2.1.2} replaced by
\begin{equation}
	\int_1^\infty
	(
		\zeta^{\lambda_2 \lambda_1}
		g_2 (\zeta)
		g_1^{\lambda_2} (\zeta)
	)^{
		- 1 / (m_2 + \lambda_2 m_1)
	}
	\zeta^{
		1 / (m_2 + \lambda_2 m_1) - 1
	}
	d\zeta
	<
	\infty
	\label{R2.1.1}
\end{equation}
and
\begin{equation}
	\int_0^1
	r^{(n - m_1) \mu_2 \mu_1 - m_2 \mu_1 - n - 1}
	h_2^{\mu_1} (r)
	h_1 (r)
	\,
	dr
	=
	\infty,
	\label{R2.1.2}
\end{equation}
respectively. 
In its turn, Theorem~\ref{T2.2} and Corollary~\ref{C2.2} remain valid with the inequality $m_1 \ge n$ replaced by $m_2 \ge n$ and condition~\eqref{T2.1.1} replaced by~\eqref{R2.1.2}.
\end{Remark}

\begin{Corollary}\label{C2.3}
Let $\lambda_1 \lambda_2 > 1$ and $\mu_1 = \mu_2 = 1$. 
Then all solutions of~\eqref{1.1} are trivial.
\end{Corollary}

\begin{Corollary}\label{C2.4}
Let $\lambda_1 = \lambda_2 = \mu_1 = \mu_2 = 1$ and
\begin{equation}
	\int_1^\infty
	\frac{
		d\zeta
	}{
		(g_1 (\zeta) g_2 (\zeta))^{1 / (m_1 + m_2)}	
		\zeta
	}
	<
	\infty.
	\label{C2.4.1}
\end{equation}
Then all solutions of~\eqref{1.1} are trivial.
\end{Corollary}

The proof of Theorems~\ref{T2.1}--\ref{T2.3} and Corollaries~\ref{C2.1}--\ref{C2.4} is given in Section~\ref{proof}. 
Now we demonstrate some applications of these statements.

\begin{Example}\label{E2.1}
Consider the system
\begin{equation}
	\left\{
		\begin{aligned}
			&
			\sum_{|\alpha| = m_1}
			\partial^\alpha
			a_\alpha (x, u_1)
			\ge
			|u_2|^{\lambda_1}
			&
			\mbox{in } \mathbb R^n,
			\\
			&
			\sum_{|\alpha| = m_2}
			\partial^\alpha
			b_\alpha (x, u_2)
			\ge
			|u_1|^{\lambda_2}
			&
			\mbox{in } \mathbb R^n.
		\end{aligned}
	\right.
	\label{E2.1.1}
\end{equation}
It does not present any particular problem to verify that both conditions~\eqref{T2.1.1} and~\eqref{T2.1.2} take the form $\lambda_1 \lambda_2 > 1$. 
In so doing,~\eqref{T2.1.2} and~\eqref{R2.1.2} are equivalent to
\begin{equation}
	(n - m_2) \lambda_1 \lambda_2 - m_1 \lambda_2 - n \le 0
	\label{E2.1.2}
\end{equation}
and
\begin{equation}
	(n - m_1) \lambda_2 \lambda_1 - m_2 \lambda_1 - n \le 0,
	\label{E2.1.3}
\end{equation}
respectively.
Thus, if $\lambda_1 \lambda_2 > 1$ and at least one of inequalities~\eqref{E2.1.2}, \eqref{E2.1.3} is valid, then in accordance with Corollary~\ref{C2.1} all solutions of~\eqref{E2.1.1} are trivial.

The above blow-up conditions are exact. We also note that these conditions coincide with those given in~\cite[Theorem~19.1]{MPbook}.
\end{Example}

\begin{Example}\label{E2.2}
Let us examine the critical exponents $\lambda_1 \ge 1$ and $\lambda_2 \ge 1$ in the condition $\lambda_1 \lambda_2 > 1$.
Consider the system
\begin{equation}
	\left\{
		\begin{aligned}
			&
			\sum_{|\alpha| = m_1}
			\partial^\alpha
			a_\alpha (x, u_1)
			\ge
			|u_2| \log^{\nu_1} (e + |u_2|)
			&
			\mbox{in } \mathbb R^n,
			\\
			&
			\sum_{|\alpha| = m_2}
			\partial^\alpha
			b_\alpha (x, u_2)
			\ge
			|u_1| \log^{\nu_2} (e + |u_1|)
			&
			\mbox{in } \mathbb R^n,
		\end{aligned}
	\right.
	\label{E2.2.1}
\end{equation}
where $\nu_i \ge 0$, $i = 1,2$, are real numbers. By Theorem~\ref{T2.1}, if
\begin{equation}
	\nu_1 + \nu_2 > m_1 + m_2,
	\label{E2.2.2}
\end{equation}
then all solutions of~\eqref{E2.2.1} are trivial. Condition~\eqref{E2.2.2} is exact. Indeed, assume that
$$
	\nu_1 + \nu_2 \le m_1 + m_2.
$$
It this case, either $\nu_1 \le m_1$ or $\nu_2 \le m_2$.
Let us assume for certainty that $\nu_1 \le m_1$.
Putting
$$
	u_1 (x_1) 
	= 
	|x_1|^{m_1}
	e^{
		e^{|x_1|}
	}
$$
and
$$
	u_2 (x_1) 
	=
	|x_1|^{m_2}
	e^{
		(m_1 - \nu_1 + m_2) 
		|x_1|
	}
	e^{
		e^{|x_1|}
	},
$$
we obtain non-trivial solutions of the system
$$
	\left\{
		\begin{aligned}
			&
			\frac{
				\partial^{m_1} (\operatorname{sign}	 x_1)^{m_1} a u_1
			}{
				\partial x_1^{m_1}
			}
			\ge
			|u_2| \log^{\nu_1} (e + |u_2|)
			&
			\mbox{in } \mathbb R^n,
			\\
			&
			\frac{
				\partial^{m_2} (\operatorname{sign}	 x_1)^{m_2} b u_2
			}{
				\partial x_1^{m_2}
			}
			\ge
			|u_1| \log^{\nu_2} (e + |u_1|)
			&
			\mbox{in } \mathbb R^n,
		\end{aligned}
	\right.
$$
where $a > 0$ and $b > 0$ are sufficiently large real numbers.
\end{Example}

\section{Proof of the main results}\label{proof}

In this section, by $C$ and $\sigma$ we denote various positive constants that can depend only on $A$, 
$n$, $m_i$, $\lambda_i$, $\mu_i$, and on 
the functions $g_i$ and $h_i$.
Let us also denote by $B_r$ the open ball of radius $r > 0$ centered at zero. In so doing, by $|B_r|$ we mean the $n$-dimensional volume of this ball.

\begin{Lemma}\label{L3.1}
Let $u_1$ and $u_2$ be solutions of system~\eqref{1.1}, then
\begin{equation}
	\int_{
		B_{r_2}
		\setminus
		B_{r_1}
	}
	|u_1|
	\,
	dx
	\ge
	C
	(r_2 - r_1)^{m_1}
	\int_{
		B_{r_1}
	}
	f_1 (|u_2|)
	\,
	dx
	\label{L3.1.1}
\end{equation}
and
\begin{equation}
	\int_{
		B_{r_2}
		\setminus
		B_{r_1}
	}
	|u_2|
	\,
	dx
	\ge
	C
	(r_2 - r_1)^{m_2}
	\int_{
		B_{r_1}
	}
	f_2 (|u_1|)
	\,
	dx
	\label{L3.1.2}
\end{equation}
for all real numbers $0 < r_1 < r_2$ such that $r_2 \le 2 r_1$. 
\end{Lemma}

\begin{proof}
We put
$$
	\varphi (x)
	=
	\varphi_0
	\left(
		\frac{r_2 - |x|}{r_2 - r_1}
	\right),
$$
where $\varphi_0 \in C^\infty ({\mathbb R})$ is a non-negative function satisfying the conditions
$$
	\left.
		\varphi_0
	\right|_{
		(- \infty, 0]
	}
	=
	0
	\quad
	\mbox{and}
	\quad
	\left.
		\varphi_0
	\right|_{
		[1, \infty)
	}
	=
	1.
$$
It is easy to see that
$$
	\left|
		\int_{\mathbb R^n}
		\sum_{|\alpha| = m_1}
		a_\alpha (x, u_1)
		\partial^\alpha
		\varphi
		\,
		dx
	\right|
	\le
	\frac{
		\sigma
	}{
		(r_2 - r_1)^{m_1}
	}
	\int_{
		B_{r_2}
		\setminus
		B_{r_1}
	}
	|u_1|
	\,
	dx
$$
and
$$
	\int_{\mathbb R^n}
	f_1 (u_2)
	\varphi
	\,
	dx
	\ge
	\int_{
		B_{r_1}
	}
	f_1 (|u_2|)
	\,
	dx.
$$
Thus, taking $\varphi$ as a test function in~\eqref{1.2}, we obtain~\eqref{L3.1.1}.
Repeating the previous argument with~\eqref{1.2} replaced by~\eqref{1.3}, we arrive at~\eqref{L3.1.2}.
\end{proof}

\begin{Lemma}\label{L3.2}
Let $u_1$ and $u_2$ be solutions of~\eqref{1.1} and $R > 0$ be a real number.
Then
\begin{equation}
	I_1 (r_2) - I_1 (r_1)
	\ge
	C 
	(r_2 - r_1)^{m_1}
	f_1 (I_2 (r_1))
	\label{L3.2.1}
\end{equation}
and
\begin{equation}
	I_2 (r_2) - I_2 (r_1)
	\ge
	C 
	(r_2 - r_1)^{m_2}
	f_2 (I_1 (r_1))
	\label{L3.2.2}
\end{equation}
for all real numbers $R \le r_1 < r_2 \le 2 R$,
where
\begin{equation}
	I_i (r)
	=
	\frac{
		1
	}{
		|B_{2 R}|
	}
	\int_{
		B_r
	}
	|u_i|
	\,
	dx,
	\quad
	i = 1,2.
	\label{L3.2.3}
\end{equation}
\end{Lemma}

\begin{proof}
From formula~\eqref{L3.1.1} of Lemma~\ref{L3.1}, it follows that
\begin{equation}
	\frac{
		1
	}{
		|B_{2 R}|
	}
	\int_{
		B_{r_2}
		\setminus
		B_{r_1}
	}
	|u_1|
	\,
	dx
	\ge
	\frac{
		C
		(r_2 - r_1)^{m_1}
	}{
		|B_{2 R}|
	}
	\int_{
		B_{r_1}
	}
	f_1 (|u_2|)
	\,
	dx.
	\label{PL3.2.1}
\end{equation}
Since $f_1$ is a non-decreasing convex function, we have
$$
	\frac{
		1
	}{
		|B_{r_1}|
	}
	\int_{
		B_{r_1}
	}
	f_1 (|u_2|)
	\,
	dx
	\ge
	f_1 
	\left(
		\frac{
			1
		}{
			|B_{r_1}|
		}
		\int_{
			B_{r_1}
		}
		|u_2|
		\,
		dx
	\right)
	\ge
	f_1 
	\left(
		\frac{
			1
		}{
			|B_{2 R}|
		}
		\int_{
			B_{r_1}
		}
		|u_2|
		\,
		dx
	\right).
$$
According to~\eqref{PL3.2.1}, this implies~\eqref{L3.2.1}.
By similar reasoning, formula~\eqref{L3.1.2} of Lemma~\ref{L3.1} leads to~\eqref{L3.2.2}.
\end{proof}

We also need the following simple statement.

\begin{Lemma}\label{L3.3}
The function $g_i$ and $h_i$ satisfy the estimates
\begin{equation}
	g_i (\zeta) 
	\le 
	C 
	\log^\sigma \zeta
	\quad
	\mbox{and}
	\quad
	h_i 
	\left( 
		\frac{1}{\zeta} 
	\right) 
	\ge 
	C 
	\log^{- \sigma} \zeta
	\label{L3.3.1}
\end{equation}
for all $\zeta \in [e, \infty)$, $i = 1,2$.
\end{Lemma}

\begin{proof}
For any $\zeta \in [e, \infty)$ we have
$$
	\zeta 
	\le 
	e^{
		e^{[\log \log \zeta] + 1}
	},
$$
where $[\log \log \zeta]$ is the integer part of $\log \log \zeta$.
In so doing, by~\eqref{1.6},
$$
	g_i 
	\left(
		e^{
			e^{[\log \log \zeta] + 1}
		}
	\right)
	\le
	C
	g_i 
	\left(
		e^{
			e^{[\log \log \zeta]}
		}
	\right)
	\le
	C^2
	g_i 
	\left(
		e^{
			e^{[\log \log \zeta] - 1}
		}
	\right)
	\le
	\ldots
	\le
	C^{[\log \log \zeta] + 1}
	g_i (e),
$$
whence in accordance with the monotonicity of the functions $g_i$ and the fact that
$$
	C^{[\log \log \zeta] + 1}
	\le
	C^{\log \log \zeta + 1}
	=
	C
	\log^{\log C} \zeta
$$
we readily obtain the first inequality in~\eqref{L3.3.1}.

In its turn,~\eqref{1.7} yields
$$
	h_i 
	\left(
		e^{
			- e^{[\log \log \zeta] + 1}
		}
	\right)
	\ge
	C
	h_i 
	\left(
		e^{
			- e^{[\log \log \zeta]}
		}
	\right)
	\ge
	C^2
	h_i 
	\left(
		e^{
			- e^{[\log \log \zeta] - 1}
		}
	\right)
	\ge
	\ldots
	\ge
	C^{[\log \log \zeta] + 1}
	h_i \left( \frac{1}{e} \right).
$$
This implies the second inequality in~\eqref{L3.3.1}.
\end{proof}

\begin{Lemma}\label{L3.4}
Let $u_1$ and $u_2$ be non-trivial solutions of~\eqref{1.1}. If~\eqref{T2.1.1} is valid, then
\begin{equation}
	\lim_{R \to \infty}
	\frac{
		1
	}{
		R^n
	}
	\int_{
		B_R
	}
	|u_1|
	\,
	dx
	=
	0.
	\label{L3.4.1}
\end{equation}
\end{Lemma}

\begin{proof}
We extend the functions $g_i$ to the whole interval $(0, \infty)$ by putting
$$
	g_i (\zeta)
	=
	\min
	\left\{
		g_i (1),
		\inf_{s \in [\zeta, 1]}
		\frac{
			f_i (s)
		}{
			s^{\lambda_i}
		}
	\right\}
	\quad
	\mbox{for all } \zeta \in (0, 1],
	\quad
	i = 1,2.
$$
In this case, inequality~\eqref{1.4} should obviously be fulfilled for all $\zeta \in (0, \infty)$. In so doing, the functions $g_i$ remain continuous and non-decreasing on the entire interval $(0, \infty)$. 

Denote
$$
	E_i (r)
	=
	\int_{B_r}
	|u_i|
	\,
	dx,
	\quad
	i = 1,2.
$$
If one of the functions $u_1$ or $u_2$ is equal to zero almost everywhere in $\mathbb R^n$, then in accordance with the definition of solutions of~\eqref{1.1} the second of these function is also equal to zero almost everywhere in $\mathbb R^n$. 
Therefore, there is a real number 

Let $R \in [R_0, \infty)$ and $l$ be the minimal positive integer such that $2^l I_1 (R) \ge I_1 (2 R)$,
where the functions $I_1$ is defined by~\eqref{L3.2.3}.
Since $I_1 (R) > 0$, such an integer $l$ obviously exists.
We put
$r_0 = R$, 
$$
	r_i 
	=
	\sup 
	\{
		r \in (r_{i - 1}, 2 R)
		:
		I_1 (r) \le 2 I_1 (r_{i - 1})
	\},
	\quad
	0 < i \le l - 1,
$$
and  $r_l = 2 R$. It can be seen that
\begin{align}
	I_1 (r_{i+1}) - I_1 (r_i)
	\ge
	{}
	&
	C
	(r_{i+1} - r_i)^{m_1 + \lambda_1 m_2}
	I_1^{\lambda_1 \lambda_2} (r_i)
	g_2^{\lambda_1} (I_1 (r_i))
	\nonumber
	\\
	&
	{}
	\times
	g_1 (\sigma (r_{i+1} - r_i)^{m_2} I_1^{\lambda_2} (r_i) g_2 (I_1 (r_i)))
	\label{PL3.4.1}
\end{align}
for all $0 \le i \le l - 1$.
Indeed, by Lemma~\ref{L3.2}, we have
$$
	I_1 (r_{i+1}) - I_1 (\rho_i)
	\ge
	C 
	(r_{i+1} - \rho_i)^{m_1}
	f_1 (I_2 (\rho_i))
$$
and
$$
	I_2 (\rho_i) - I_2 (r_1)
	\ge
	C 
	(\rho_i - r_i)^{m_2}
	f_2 (I_1 (r_i)),
$$
where $\rho_i = (r_{i+1} + r_i) / 2$. 
This implies the estimate
$$
	I_1 (r_{i+1}) - I_1 (\rho_i)
	\ge
	C 
	(r_{i+1} - \rho_i)^{m_1}
	f_1 (
		\sigma 
		(\rho_i - r_i)^{m_2}
		f_2 (I_1 (r_i))
	),
$$
whence in accordance with~\eqref{1.4} we arrive at~\eqref{PL3.4.1}.

Take a real number $\varepsilon \in (0, 1 / n) \cap (0, \lambda_2 / m_2)$.
We denote by $\Xi_1$ the set of integers $0 \le i \le l - 1$ for which
$
	r_{i+1} - r_i < I_1^{- \varepsilon} (r_i).
$
In so doing, let $\Xi_2$ be the set of all other integers $0 \le i \le l - 1$.
From the definition of real numbers $r_i$, it follows that 
$$
	I_1 (r_i) \le 2^i I_1 (R),
	\quad
	i = 0, \ldots, l - 1;
$$
therefore,
\begin{equation}
	\sum_{i \in \Xi_1}
	(r_{i+1} - r_i)
	<
	\sum_{i \in \Xi_1}
	2^{- \varepsilon i}
	I_1^{- \varepsilon} (R)
	\le
	\frac{
		I_1^{- \varepsilon} (R)
	}{
		1 - 2^{- \varepsilon}
	}.
	\label{PL3.4.2}
\end{equation}
At the same time, taking into account~\eqref{L3.2.3}, we obtain
$$
	I_1^{- \varepsilon} (R)
	=
	\left(
		\frac{
			|B_{2 R}|
		}{
			E_1 (R)
		}
	\right)^\varepsilon
	\le
	\left(
		\frac{
			|B_2|
		}{
			E_1 (R_0)
		}
	\right)^\varepsilon
	R^{\varepsilon n}.
$$
Since $\varepsilon n < 1$, this allows us to assert that
\begin{equation}
	\frac{
		I_1^{- \varepsilon} (R)
	}{
		1 - 2^{- \varepsilon}
	}
	\le
	\frac{R}{2}
	\label{PL3.4.3}
\end{equation}
if $R$ is large enough.
Without loss of generality, one can assume that the real number $R_0$ is chosen to be so large that~\eqref{PL3.4.3} is valid for all $R \ge R_0$.
Thus, combining~\eqref{PL3.4.3} with~\eqref{PL3.4.2}, we arrive at the inequality
\begin{equation}
	\sum_{i \in \Xi_1}
	(r_{i+1} - r_i)
	<
	\frac{R}{2}.
	\label{PL3.4.8}
\end{equation}
This, in turn, yields
\begin{equation}
	\sum_{i \in \Xi_2}
	(r_{i+1} - r_i)
	>
	\frac{R}{2}.
	\label{PL3.4.4}
\end{equation}

We have
$
	I_1 (r_i) \ge I_1 (R)
$
and
$$
	(r_{i+1} - r_i)^{m_2} 
	I_1^{\lambda_2} (r_i)
	\ge
	I_1^{\lambda_2 - \varepsilon m_2} (r_i)
$$
for all $i \in \Xi_2$. Hence,~\eqref{PL3.4.1} implies estimate
\begin{align*}
	I_1 (r_{i+1}) - I_1 (r_i)
	\ge
	{}
	&
	C
	(r_{i+1} - r_i)^{m_1 + \lambda_1 m_2}
	I_1^{\lambda_1 \lambda_2} (r_i)
	g_2^{\lambda_1} (I_1 (r_i))
	\nonumber
	\\
	&
	{}
	\times
	g_1 
	(
		\sigma 
		I_1^{\lambda_2 - \varepsilon m_2} (r_i) 
		g_2 
		(
			I_1 (R)
		)
	)
\end{align*}
or, in other words,
\begin{align}
	\left(
		\frac{
			I_1 (r_{i+1}) - I_1 (r_i)
		}{
			I_1^{\lambda_1 \lambda_2} (r_i)
			g_2^{\lambda_1} (I_1 (r_i))
			g_1 
			(
				\sigma 
				I_1^{\lambda_2 - \varepsilon m_2} (r_i) 
				g_2 
				(
					I_1 (R)
				)
			)
		}
	\right)^{1 / (m_1 + \lambda_1 m_2)}
	\ge
	C (r_{i+1} - r_i)
	\label{PL3.4.5}
\end{align}
for all $i \in \Xi_2$. 
From the definition of the real numbers $r_i$, it follows that $I_1 (r_{i+1}) = 2 I_1 (r_i)$ for all $0 \le i < l - 1$ and, moreover, $I_1 (r_l) \le 2 I_1 (r_{l-1})$. Thus,
\begin{align*}
	&
	\int_{
		I_1 (r_i)
	}^{
		2 I_1 (r_i)
	}
	(
		\zeta^{\lambda_1 \lambda_2}
		g_1 (
			\sigma 
			\zeta^{\lambda_2 - \varepsilon m_2}
			g_2 (I_1 (R))
		)
		g_2^{\lambda_1} (\zeta / 2)
	)^{
		- 1 / (m_1 + \lambda_1 m_2)
	}
	\zeta^{
		1 / (m_1 + \lambda_1 m_2) - 1
	}
	d\zeta
	\\
	&
	\quad
	{}
	\ge
	C
	\left(
		\frac{
			I_1 (r_{i+1}) - I_1 (r_i)
		}{
			I_1^{\lambda_1 \lambda_2} (r_i)
			g_2^{\lambda_1} (I_1 (r_i))
			g_1 
			(
				\sigma 
				I_1^{\lambda_2 - \varepsilon m_2} (r_i) 
				g_2 
				(
					I_1 (R)
				)
			)
		}
	\right)^{1 / (m_1 + \lambda_1 m_2)},
	\quad
	i = 0, \ldots, l - 1,
\end{align*}
whence in accordance with~\eqref{PL3.4.5}, it follows that
$$
	\int_{
		I_1 (r_i)
	}^{
		2 I_1 (r_i)
	}
	(
		\zeta^{\lambda_1 \lambda_2}
		g_1 (
			\sigma 
			\zeta^{\lambda_2 - \varepsilon m_2}
			g_2 (I_1 (R))
		)
		g_2^{\lambda_1} (\zeta / 2)
	)^{
		- 1 / (m_1 + \lambda_1 m_2)
	}
	\zeta^{
		1 / (m_1 + \lambda_1 m_2) - 1
	}
	d\zeta
	\ge
	C (r_{i+1} - r_i)
$$
for all $i \in \Xi_2$. 
Summing this over all $i \in \Xi_2$, we have
\begin{align*}
	&
	\int_{
		E_1 (R) / |B_{2 R}|
	}^\infty
	\left(
		\zeta^{\lambda_1 \lambda_2}
		g_1 
		\left(
			\sigma 
			\zeta^{\lambda_2 - \varepsilon m_2}
			g_2 
			\left(
				\frac{
					E_1 (R)
				}{
					|B_{2 R}|
				}
			\right)
		\right)
		g_2^{\lambda_1} 
		\left(
			\frac{\zeta}{2}
		\right)
	\right)^{
		- 1 / (m_1 + \lambda_1 m_2)
	}
	\zeta^{
		1 / (m_1 + \lambda_1 m_2) - 1
	}
	d\zeta
	\\
	&
	\quad
	{}
	\ge
	C \sum_{i \in \Xi_2} 
	(r_{i+1} - r_i).
\end{align*}
In view of~\eqref{PL3.4.4}, the last estimate yields
\begin{align}
	&
	\int_{
		E_1 (R) / |B_{2 R}|
	}^\infty
	\left(
		\zeta^{\lambda_1 \lambda_2}
		g_1 
		\left(
			\sigma 
			\zeta^{\lambda_2 - \varepsilon m_2}
			g_2 
			\left(
				\frac{
					E_1 (R)
				}{
					|B_{2 R}|
				}
			\right)
		\right)
		g_2^{\lambda_1} 
		\left(
			\frac{\zeta}{2}
		\right)
	\right)^{
		- 1 / (m_1 + \lambda_1 m_2)
	}
	\zeta^{
		1 / (m_1 + \lambda_1 m_2) - 1
	}
	d\zeta
	\nonumber
	\\
	&
	\quad
	{}
	\ge
	C R.
	\label{PL3.4.6}
\end{align}

Assume on the contrary that~\eqref{L3.4.1} is not satisfied. Then there exists a sequence of real numbers $R_i \ge R_0$, $i = 1,2,\ldots$, such that $R_i \to \infty$ as $i \to \infty$ and, moreover,
$$
	\frac{
		E_1 (R_i)
	}{
		|B_{2 R_i}|
	}
	\ge
	\delta
$$
with some real number $\delta > 0$ for all $i = 1,2,\ldots$.
Taking $R = R_i$ in~\eqref{PL3.4.6}, we obtain
\begin{equation}
	\int_\delta^\infty
	(
		\zeta^{\lambda_1 \lambda_2}
		g_1 (
			\sigma 
			\zeta^{\lambda_2 - \varepsilon m_2}
			g_2 (\delta)
		)
		g_2^{\lambda_1} (\zeta / 2)
	)^{
		- 1 / (m_1 + \lambda_1 m_2)
	}
	\zeta^{
		1 / (m_1 + \lambda_1 m_2) - 1
	}
	d\zeta
	\ge
	C R_i
	\label{PL3.4.7}
\end{equation}
for all $i = 1,2,\ldots$.
From~\eqref{1.6}, it follows that
$$
	g_1 (
		\sigma 
		\zeta^{\lambda_2 - \varepsilon m_2}
		g_2 (\delta)
	)
	g_2^{\lambda_1} (\zeta / 2)
	\ge
	C
	g_1 (\zeta)
	g_2^{\lambda_1} (\zeta)
$$
for all sufficiently large $\zeta > 0$.
Thus, according to~\eqref{T2.1.1}, we have
$$
	\int_\delta^\infty
	(
		\zeta^{\lambda_1 \lambda_2}
		g_1 (
			\sigma 
			\zeta^{\lambda_2 - \varepsilon m_2}
			g_2 (\delta)
		)
		g_2^{\lambda_1} (\zeta / 2)
	)^{
		- 1 / (m_1 + \lambda_1 m_2)
	}
	\zeta^{
		1 / (m_1 + \lambda_1 m_2) - 1
	}
	d\zeta
	<
	\infty.
$$
This obviously contradicts~\eqref{PL3.4.7}.
\end{proof}

\begin{Lemma}\label{L3.5}
Let $u_1$ and $u_2$ be non-trivial solutions of~\eqref{1.1}. If~\eqref{R2.1.1} is valid, then
$$
	\lim_{R \to \infty}
	\frac{
		1
	}{
		R^n
	}
	\int_{
		B_R
	}
	|u_2|
	\,
	dx
	=
	0.
$$
\end{Lemma}

The proof is similar to the proof of Lemma~\ref{L3.4}.

\begin{Lemma}\label{L3.6}
Let $u_1$ and $u_2$ be solutions of~\eqref{1.1}. Then
\begin{equation}
	J_1 (r_2) - J_1 (r_1)
	\ge
	C (r_2^n - r_1^n)
	f_2 
	\left(
		\frac{
			\sigma
			(r_2 - r_1)^{m_1}
		}{
			r_2^n - r_1^n
		}
		J_2 (r_1)
	\right)
	\label{L3.6.1}
\end{equation}
and
\begin{equation}
	J_2 (r_2) - J_2 (r_1)
	\ge
	C (r_2^n - r_1^n)
	f_1 
	\left(
		\frac{
			\sigma
			(r_2 - r_1)^{m_2}
		}{
			r_2^n - r_1^n
		}
		J_1 (r_1)
	\right)
	\label{L3.6.2}
\end{equation}
for all real numbers $0 < r_1 < r_2$, where
\begin{equation}
	J_1 (r)
	=
	\int_{B_r}
	f_2 (|u_1|)
	\,
	dx
	\quad
	\mbox{and}
	\quad
	J_2 (r)
	=
	\int_{B_r}
	f_1 (|u_2|)
	\,
	dx.
	\label{L3.6.3}
\end{equation}
\end{Lemma}

\begin{proof}
Taking into account formula~\eqref{L3.1.1} of Lemma~\ref{L3.1}, we have
$$
	\frac{
		1
	}{
		|B_{r_2} \setminus B_{r_1}|
	}
	\int_{
		B_{r_2}
		\setminus
		B_{r_1}
	}
	|u_1|
	\,
	dx
	\ge
	\frac{
		C
		(r_2 - r_1)^{m_1}
	}{
		|B_{r_2} \setminus B_{r_1}|
	}
	\int_{
		B_{r_1}
	}
	f_1 (|u_2|)
	\,
	dx.
$$
Since $f_2$ is a non-decreasing function, this implies that
$$
	f_2
	\left(
		\frac{
			1
		}{
			|B_{r_2} \setminus B_{r_1}|
		}	
		\int_{
			B_{r_2}
			\setminus
			B_{r_1}
		}
		|u_1|
		\,
		dx
	\right)
	\ge
	f_2
	\left(
		\frac{
			C
			(r_2 - r_1)^{m_1}
		}{
			|B_{r_2} \setminus B_{r_1}|
		}
		\int_{
			B_{r_1}
		}
		f_1 (|u_2|)
		\,
		dx
	\right),
$$
whence in accordance with the inequality
$$
	\frac{
		1
	}{
		|B_{r_2} \setminus B_{r_1}|
	}	
	\int_{
		B_{r_2}
		\setminus
		B_{r_1}
	}
	f_2 (|u_1|)
	\,
	dx
	\ge
	f_2
	\left(
		\frac{
			1
		}{
			|B_{r_2} \setminus B_{r_1}|
		}	
		\int_{
			B_{r_2}
			\setminus
			B_{r_1}
		}
		|u_1|
		\,
		dx
	\right)
$$
which is valid as $f_2$ is a convex function, it follows that
$$
	\frac{
		1
	}{
		|B_{r_2} \setminus B_{r_1}|
	}	
	\int_{
		B_{r_2}
		\setminus
		B_{r_1}
	}
	f_2 (|u_1|)
	\,
	dx
	\ge
	f_2
	\left(
		\frac{
			C
			(r_2 - r_1)^{m_1}
		}{
			|B_{r_2} \setminus B_{r_1}|
		}
		\int_{
			B_{r_1}
		}
		f_1 (|u_2|)
		\,
		dx
	\right).
$$
The last estimate is obviously equivalent to~\eqref{L3.6.1}.
In its turn, repeating the above argument with~\eqref{L3.1.1} replaced by~\eqref{L3.1.2}, we obtain~\eqref{L3.6.2}.
\end{proof}

\begin{Lemma}\label{L3.7}
Let $u_1$ and $u_2$ be non-trivial solutions of~\eqref{1.1} and, moreover,~\eqref{T2.1.1} hold. 
If $\mu_1 \mu_2 > 1$, then 
\begin{equation}
	J_1^{1 - \mu_1 \mu_2} (R)
	-
	J_1^{1 - \mu_1 \mu_2} (2 R)
	\ge
	C
	R^{n + m_1 \mu_2 - (n - m_2) \mu_1 \mu_2}
	h_1^{\mu_2} 
	\left(
		\frac{1}{R}
	\right)
	h_2
	\left(
		\frac{1}{R}
	\right)
	\label{L3.7.1}
\end{equation}
for all sufficiently large $R > 0$,
where the function $J_1$ is defined by~\eqref{L3.6.3}.
\end{Lemma}

\begin{proof}
Our reasoning is in many ways similar to that given in the proof of Lemma~\ref{L3.6}.
Take a real number $R > 0$ such that $J_1 (R) > 0$.
Since $u_1$ and $u_2$ are nontrivial solutions~\eqref{1.1}, the last inequality must obviously hold for all $R$ in a neighborhood of infinity. Assume further that $l$ is the minimal positive integer satisfying the condition $2^l J_1 (R) \ge J_1 (2 R)$.
We put
$r_0 = R$, 
$$
	r_i 
	=
	\sup 
	\{
		r \in (r_{i - 1}, 2 R)
		:
		J_1 (r) \le 2 J_1 (r_{i - 1})
	\},
	\quad
	0 < i \le l - 1,
$$
and  $r_l = 2 R$. 
Let us show that
\begin{align}
	&
	J_1 (r_{i+1}) - J_1 (r_i)
	\ge
	C
	(r_{i+1} - r_i)^{1 + m_1 \mu_2 + (m_2 - 1) \mu_1 \mu_2}
	r_i^{(1 - \mu_1 \mu_2)(n - 1)}
	\nonumber
	\\
	&
	\quad
	{}
	\times
	h_2
	\left(
		\sigma
		(r_{i+1} - r_i)^{m_1 + m_2 \mu_1 - 1}
		\left(
			\frac{
				J_1 (r_i)
			}{
				r_i^{n - 1}
			}
		\right)^{\mu_1}
		h_1
		\left(
			\sigma
			(r_{i+1} - r_i)^{m_2 - 1}
			\frac{
				J_1 (r_i)
			}{
				r_i^{n - 1}
			}
		\right)
	\right)
	\nonumber
	\\
	&
	\quad
	{}
	\times
	h_1^{\mu_2}
	\left(
		\sigma
		(r_{i+1} - r_i)^{m_2 - 1}
		\frac{
			J_1 (r_i)
		}{
			r_i^{n - 1}
		}
	\right)
	J_1^{\mu_1 \mu_2} (r_i)
	\label{PL3.7.1}
\end{align}
for all $0 \le i \le l - 1$.
Indeed, by Lemma~\ref{L3.6}, we have
\begin{equation}
	J_1 (r_{i+1}) - J_1 (\rho_i)
	\ge
	C 
	(r_{i+1}^n - \rho_i^n)
	f_2 
	\left(
		\frac{
			\sigma
			(r_{i+1} - \rho_i)^{m_1}
		}{
			r_{i+1}^n - \rho_i^n
		}
		J_2 (\rho_i)
	\right)
	\label{PL3.7.2}
\end{equation}
and
\begin{equation}
	J_2 (\rho_i) - J_2 (r_i)
	\ge
	C 
	(\rho_i^n - r_i^n)
	f_1 
	\left(
		\frac{
			\sigma
			(\rho_i - r_i)^{m_2}
		}{
			\rho_i^n - r_i^n
		}
		J_1 (r_i)
	\right)
	\label{PL3.7.3}
\end{equation}
for all $0 \le i \le l - 1$,
where $\rho_i = (r_{i+1} + r_i) / 2$.

At the same time, applying Lemma~\ref{L3.1} with $r_2 = 4 R$ and $r_1 = 2 R$, we obtain
\begin{equation}
	\int_{
		B_{4 R}
		\setminus
		B_{2 R}
	}
	|u_1|
	\,
	dx
	\ge
	C
	R^{m_1}
	\int_{
		B_{2 R}
	}
	f_1 (|u_2|)
	\,
	dx,
	\label{PL3.7.4}
\end{equation}
whence it follows that
$$
	\frac{1}{R^n}
	\int_{
		B_{4 R}
		\setminus
		B_{2 R}
	}
	|u_1|
	\,
	dx
	\ge
	\frac{
		C
	}{
			R^{n - m_1}
	}
	\int_{
		B_{2 R}
	}
	f_1 (|u_2|)
	\,
	dx
	\ge
	\frac{
		\sigma
		(r_{i+1} - \rho_i)^{m_1}
	}{
		r_{i+1}^n - \rho_i^n
	}
	J_2 (\rho_i)
$$
for all $0 \le i \le l - 1$.
In view of Lemma~\ref{L3.4},
\begin{equation}
	\lim_{R \to \infty}
	\frac{
		1
	}{
		R^n
	}
	\int_{
		B_{4 R}
		\setminus
		B_{2 R}
	}
	|u_1|
	\,
	dx
	=
	0;
	\label{PL3.7.5}
\end{equation}
therefore, if the real number $R > 0$ is large enough, 
then in accordance with~\eqref{1.5} we can estimate the last factor on the right in~\eqref{PL3.7.2} as follows:
$$
	f_2 
	\left(
		\frac{
			\sigma
			(r_{i+1} - \rho_i)^{m_1}
		}{
			r_{i+1}^n - \rho_i^n
		}
		J_2 (\rho_i)
	\right)
	\ge
	C
	\left(
		\frac{
			(r_{i+1} - \rho_i)^{m_1}
		}{
			r_{i+1}^n - \rho_i^n
		}
		J_2 (\rho_i)
	\right)^{\mu_2}
	h_2
	\left(
		\frac{
			\sigma
			(r_{i+1} - \rho_i)^{m_1}
		}{
			r_{i+1}^n - \rho_i^n
		}
		J_2 (\rho_i)
	\right).
$$
Hence,~\eqref{PL3.7.2} implies the inequality
$$
	J_1 (r_{i+1}) - J_1 (\rho_i)
	\ge
	C 
	(r_{i+1}^n - \rho_i^n)
	\left(
		\frac{
			(r_{i+1} - \rho_i)^{m_1}
		}{
			r_{i+1}^n - \rho_i^n
		}
		J_2 (\rho_i)
	\right)^{\mu_2}
	h_2
	\left(
		\frac{
			\sigma
			(r_{i+1} - \rho_i)^{m_1}
		}{
			r_{i+1}^n - \rho_i^n
		}
		J_2 (\rho_i)
	\right)
$$
for all $0 \le i \le l - 1$.
Combining this with~\eqref{PL3.7.3}, we have
\begin{align}
	J_1 (r_{i+1}) - J_1 (\rho_i)
	\ge
	{}
	&
	C 
	(r_{i+1}^n - \rho_i^n)
	\left(
		(r_{i+1} - \rho_i)^{m_1}
		f_1 
		\left(
			\frac{
				\sigma
				(\rho_i - r_i)^{m_2}
			}{
				\rho_i^n - r_i^n
			}
			J_1 (r_i)
		\right)
	\right)^{\mu_2}
	\nonumber
	\\
	&
	{}
	\times
	h_2
	\left(
		\sigma
		(r_{i+1} - \rho_i)^{m_1}
		f_1 
		\left(
			\frac{
				\sigma
				(\rho_i - r_i)^{m_2}
			}{
				\rho_i^n - r_i^n
			}
			J_1 (r_i)
		\right)
	\right)
	\label{PL3.7.6}
\end{align}
for all $0 \le i \le l - 1$.
By~\eqref{PL3.7.4} and~\eqref{PL3.7.5}, 
$$
	\lim_{R \to \infty}
	\frac{
		J_2 (R)
	}{
		R^{n - m_1}
	}
	=
	0;
$$
therefore, applying formula~\eqref{L3.6.2} of Lemma~\ref{L3.6} with $r_1 = R$ and $r_2 = 2 R$, we arrive in the limit as $R \to \infty$ at the relation
\begin{equation}
	\lim_{R \to \infty}
	\frac{
		J_1 (R)
	}{
		R^{n - m_2}
	}
	=
	0.
	\label{PL3.7.12}
\end{equation}
Thus, if the real number $R > 0$ is large enough, then condition~\eqref{1.5} and the obvious inequality 
$$
	\frac{
		J_1 (2 R)
	}{
		(2 R)^{n - m_2}
	}
	\ge
	C
	\frac{
		(\rho_i - r_i)^{m_2}
	}{
		\rho_i^n - r_i^n
	}
	J_1 (r_i)
$$
allow one to assert that
$$
	f_1 
	\left(
		\frac{
			\sigma
			(\rho_i - r_i)^{m_2}
		}{
			\rho_i^n - r_i^n
		}
		J_1 (r_i)
	\right)
	\ge
	C
	\left(
		\frac{
			\sigma
			(\rho_i - r_i)^{m_2}
		}{
			\rho_i^n - r_i^n
		}
		J_1 (r_i)
	\right)^{\mu_1}
	h_1
	\left(
		\frac{
			\sigma
			(\rho_i - r_i)^{m_2}
		}{
			\rho_i^n - r_i^n
		}
		J_1 (r_i)
	\right)
$$
for all $0 \le i \le l - 1$.
Combining this with~\eqref{PL3.7.6}, we obtain~\eqref{PL3.7.1}.

Let us take a real number $\varepsilon \in (0, 1 / (m_1 + m_2 \mu_1 - 1))$.
We denote by $\Xi_1$ the set of integers $0 \le i \le l - 1$ such that
$
	r_{i+1} - r_i < R^\varepsilon J_1^{- \varepsilon} (r_i).
$
Also let $\Xi_2$ be the set of all other integers $0 \le i \le l - 1$.
It is obvious that
$$
	\sum_{i \in \Xi_1}
	(r_{i+1} - r_i)
	<
	R^\varepsilon
	\sum_{i=0}^{l-1}
	2^{-\varepsilon i}
	J_1^{- \varepsilon} (R)
	<
	\frac{
		2 
		R^\varepsilon 
		J_1^{- \varepsilon} (R)
	}{
		1 - 2^{-\varepsilon}
	}.
$$
Since $\varepsilon < 1$ and $J_1$ is a non-decreasing function, this implies inequality~\eqref{PL3.4.8}
for all $R > 0$ in a neighborhood of infinity. Hence, assuming that the real number $R > 0$ is large enough, one can assert that~\eqref{PL3.4.4} holds.

At the same time, according to~\eqref{PL3.7.1}, we have
\begin{align}
	&
	J_1 (r_{i+1}) - J_1 (r_i)
	\ge
	C
	(r_{i+1} - r_i)^{1 + m_1 \mu_2 + (m_2 - 1) \mu_1 \mu_2}
	R^{(1 - \mu_1 \mu_2)(n - 1)}
	\nonumber
	\\
	&
	\quad
	{}
	\times
	h_2
	\left(
		\sigma
		J_1^{\mu_1 - \varepsilon (m_1 + m_2 \mu_1 - 1)} (r_i)
		R^{
			\varepsilon (m_1 + m_2 \mu_1 - 1) - (n - 1) \mu_1
		}
		h_1
		\left(
			\sigma
			J_1^{1 - \varepsilon (m_2 - 1)} (r_i)
			R^{\varepsilon (m_2 - 1) - n + 1}
		\right)
	\right)
	\nonumber
	\\
	&
	\quad
	{}
	\times
	h_1^{\mu_2}
	\left(
		\sigma
		J_1^{1 - \varepsilon (m_2 - 1)} (r_i)
		R^{\varepsilon (m_2 - 1) - n + 1}
	\right)
	J_1^{\mu_1 \mu_2} (r_i)
	\label{PL3.7.7}
\end{align}
for all $i \in \Xi_2$.
In the case where the real number $R > 0$ is large enough, from condition~\eqref{1.7} and Lemma~\ref{L3.3}, it follows that
$$
	h_1
	\left(
		\sigma
		J_1^{1 - \varepsilon (m_2 - 1)} (r_i)
		R^{\varepsilon (m_2 - 1) - n + 1}
	\right)
	\ge
	h_1
	\left(
		\frac{1}{R}
	\right)
$$
and
\begin{align*}
	&
	h_2
	\left(
		\sigma
		J_1^{\mu_1 - \varepsilon (m_1 + m_2 \mu_1 - 1)} (r_i)
		R^{
			\varepsilon (m_1 + m_2 \mu_1 - 1) - (n - 1) \mu_1
		}
		h_1
		\left(
			\sigma
			J_1^{1 - \varepsilon (m_2 - 1)} (r_i)
			R^{\varepsilon (m_2 - 1) - n + 1}
		\right)
	\right)
	\\
	&
	\quad
	{}
	\ge
	h_1
	\left(
		\frac{1}{R}
	\right)
\end{align*}
for all $0 \le i \le l - 1$.
Thus, if the real number $R > 0$ is large enough, then~\eqref{PL3.7.7} implies the estimate
\begin{align}
	J_1 (r_{i+1}) - J_1 (r_i)
	\ge
	{}
	&
	C
	(r_{i+1} - r_i)^{1 + m_1 \mu_2 + (m_2 - 1) \mu_1 \mu_2}
	R^{(1 - \mu_1 \mu_2)(n - 1)}
	\nonumber
	\\
	&
	{}
	\times
	h_1^{\mu_2}
	\left(
		\frac{1}{R}
	\right)
	h_2
	\left(
		\frac{1}{R}
	\right)
	J_1^{\mu_1 \mu_2} (r_i)
	\label{PL3.7.8}
\end{align}
for all $i \in \Xi_2$.
At first, let $l - 1 \in \Xi_2$ and $r_l - r_{l-1} \ge R / 4$. 
In this case,~\eqref{PL3.7.8} yields
$$
	\frac{
		J_1 (r_l) - J_1 (r_{l-1})
	}{
		J_1^{\mu_1 \mu_2} (r_{l-1})
	}
	\ge
	C
	R^{n + m_1 \mu_2 - (n - m_2) \mu_1 \mu_2}
	h_1^{\mu_2} 
	\left(
		\frac{1}{R}
	\right)
	h_2
	\left(
		\frac{1}{R}
	\right).
$$
Since
$$
	\frac{
		J_1^{1 - \mu_1 \mu_2} (r_{l-1})
		-
		J_1^{1 - \mu_1 \mu_2} (r_l)
	}{
		\mu_1 \mu_2 - 1
	}
	=
	\int_{
		J_1 (r_{l-1})
	}^{
		J_1 (r_l)
	}
	\frac{
		d\zeta
	}{
		\zeta^{\mu_1 \mu_2}
	}
	\ge
	\frac{
		J_1 (r_l) - J_1 (r_{l-1})
	}{
		2^{\mu_1 \mu_2}
		J_1^{\mu_1 \mu_2} (r_{l-1})
	},
$$
this immediately leads to~\eqref{L3.7.1}. Now, let $l - 1 \not\in \Xi_2$ or $r_l - r_{l-1} < R / 4$. 
Then, assuming that the real number $R > 0$ is large enough, in accordance with~\eqref{PL3.4.4} 
we obtain 
\begin{equation}
	\sum_{
		i \in \Xi_2,
		\;
		i \le l - 2
	}
	(r_{i+1} - r_i)
	\ge
	\frac{R}{4}.
	\label{PL3.7.9}
\end{equation}
Hence,~\eqref{PL3.7.8} allows one to assert that
\begin{align}
	&
	\sum_{
		i \in \Xi_2,
		\;
		i \le l - 2
	}
	\left(
		\frac{
			J_1 (r_{i+1}) - J_1 (r_i)
		}{
			J_1^{\mu_1 \mu_2} (r_i)
		}
	\right)^{
		1 / (1 + m_1 \mu_2 + (m_2 - 1) \mu_1 \mu_2)
	}
	\nonumber
	\\
	&
	\quad
	{}
	\ge
	C
	\left(
		R^{n + m_1 \mu_2 - (n - m_2) \mu_1 \mu_2} 
		h_1^{\mu_2}
		\left(
			\frac{1}{R}
		\right)
		h_2
		\left(
			\frac{1}{R}
		\right)
	\right)^{1 / (1 + m_1 \mu_2 + (m_2 - 1) \mu_1 \mu_2)}.
	\label{PL3.7.10}
\end{align}
Since $J_1 (r_0) = J_1 (R)$ and $J_1 (r_{i+1}) = 2 J_1 (r_i)$ for all $0 \le i \le l - 2$, we have
$$
	J_1^{(1 - \mu_1 \mu_2) / (1 + m_1 \mu_2 + (m_2 - 1) \mu_1 \mu_2)} (R)
	\ge
	C
	\sum_{i=0}^{l-2}
	\left(
		\frac{
			J_1 (r_{i+1}) - J_1 (r_i)
		}{
			J_1^{\mu_1 \mu_2} (r_i)
		}
	\right)^{
		1 / (1 + m_1 \mu_2 + (m_2 - 1) \mu_1 \mu_2)
	};
$$
therefore,~\eqref{PL3.7.10} implies the estimate
\begin{equation}
	J_1^{1 - \mu_1 \mu_2} (R)
	\ge
	C
	R^{n + m_1 \mu_2 - (n - m_2) \mu_1 \mu_2}
	h_1^{\mu_2} 
	\left(
		\frac{1}{R}
	\right)
	h_2
	\left(
		\frac{1}{R}
	\right).
	\label{PL3.7.11}
\end{equation}
To complete the proof, it remains to note that
$
	J_1 (2R)
	\ge
	2
	J_1 (R)
$
since $l \ge 2$ in view of inequality~\eqref{PL3.7.9}. Thus,
$$
	J_1^{1 - \mu_1 \mu_2} (R)
	-
	J_1^{1 - \mu_1 \mu_2} (2R)
	\ge
	(1 - 2^{1 - \mu_1 \mu_2})
	J_1^{1 - \mu_1 \mu_2} (R)
$$
and~\eqref{L3.7.1} follows immediately from~\eqref{PL3.7.11}.
\end{proof}

\begin{Lemma}\label{L3.8}
Let $u_1$ and $u_2$ be non-trivial solutions of~\eqref{1.1} and, moreover,~\eqref{T2.1.1} hold.  If $\mu_1 = \mu_2 = 1$, then 
\begin{equation}
	J_1^{- \delta} (R)
	-
	J_1^{- \delta} (2 R)
	\ge
	C
	\delta^{m_1 + m_2 + 1}
	R^{m_1 + m_2 - \delta (n - m_2)}
	h_1
	\left(
		\frac{1}{R}
	\right)
	h_2
	\left(
		\frac{1}{R}
	\right)
	\label{L3.8.1}
\end{equation}
for all real numbers $0 < \delta < 1$ and for all sufficiently large $R > 0$,
where the function $J_1$ is defined by~\eqref{L3.6.3}.
\end{Lemma}

\begin{proof}
Take a real number $R > 0$ such that $J_1 (R) > 0$. 
Since $u_1$ and $u_2$ are non-trivial solutions of~\eqref{1.1}, the last inequality is obviously valid for all $R$ in a neighborhood of infinity.
Also let $r_i$, $i = 0, \ldots, l$, be the finite sequence of real numbers constructed in the proof of Lemma~\ref{L3.7}.

We denote by $\Xi_1$ the set of integers $0 \le i \le l - 1$ such that
$
	r_{i+1} - r_i < R^\varepsilon J_1^{- \varepsilon} (r_i),
$
where $\varepsilon \in (0, 1 / (m_1 + m_2 - 1))$ is a real number satisfying the condition
$n + m_1 - \delta (n - m_2) > 0$.
Also let $\Xi_2$ be the set of all other integers $0 \le i \le l - 1$.

For sufficiently large $R > 0$, using the argument given in the proof of~\eqref{PL3.7.8}, we obtain
\begin{equation}
	J_1 (r_{i+1}) - J_1 (r_i)
	\ge
	C
	(r_{i+1} - r_i)^{m_1 + m_2}
	h_1
	\left(
		\frac{1}{R}
	\right)
	h_2
	\left(
		\frac{1}{R}
	\right)
	J_1 (r_i)
	\label{PL3.8.1}
\end{equation}
for all $i \in \Xi_2$. 
Let $0 < \delta < 1$ be a real number.
From~\eqref{PL3.7.12}, it follows that
$$
	\frac{
		J_1 (r_i)
	}{
		r_i^{n - m_2}
	}
	\ge
	\left(
		\frac{
			J_1 (r_i)
		}{
			r_i^{n - m_2}
		}
	\right)^{1 + \delta}
$$
for all $0 \le i \le l - 1$ if $R > 0$ is large enough,
whence in accordance with~\eqref{PL3.8.1} we obtain
\begin{equation}
	J_1 (r_{i+1}) - J_1 (r_i)
	\ge
	C
	(r_{i+1} - r_i)^{m_1 + m_2}
	r_i^{- \delta (n - m_2)}
	h_1
	\left(
		\frac{1}{R}
	\right)
	h_2
	\left(
		\frac{1}{R}
	\right)
	J_1^{1 + \delta} (r_i)
	\label{PL3.8.2}
\end{equation}
for all $i \in \Xi_2$. 

At first, assume that $l - 1 \in \Xi_2$ and $r_l - r_{l-1} \ge R / 4$. In this case,~\eqref{PL3.8.2} yields
$$
	\frac{
		J_1 (r_l) - J_1 (r_{l-1})
	}{
		J_1^{1 + \delta} (r_i)
	}
	\ge
	C
	R^{m_1 + m_2 - \delta (n - m_2)}
	h_1
	\left(
		\frac{1}{R}
	\right)
	h_2
	\left(
		\frac{1}{R}
	\right).
$$
Since
$$
	\frac{
		J_1^{- \delta} (R)
		-
		J_1^{- \delta} (2 R)
	}{
		\delta
	}
	\ge
	\frac{
		J_1^{- \delta} (r_{l-1})
		-
		J_1^{- \delta} (r_l)
	}{
		\delta
	}
	=
	\int_{
		J_1 (r_{l-1})
	}^{
		J_1 (r_l)
	}
	\frac{
		d \zeta
	}{
		\zeta^{1 + \delta}
	}
	\ge
	\frac{
		J_1 (r_l) - J_1 (r_{l-1})
	}{
		2^{1 + \delta}
		J_1^{1 + \delta} (R)
	},
$$
this immediately implies~\eqref{L3.8.1}.
Now, let $l - 1 \not\in \Xi_2$ or $r_l - r_{l-1} < R / 4$. 
As in the proof of Lemma~\ref{L3.7}, we can also assume that $R > 0$ is large enough for~\eqref{PL3.4.4} to be valid.
Therefore, one can assert that~\eqref{PL3.7.9} holds.

Taking into account~\eqref{PL3.8.2} and the fact that $J_1 (r_{i+1}) = 2 J_1 (r_i)$ for all $0 \le i \le l - 2$, we have
$$
	J_1^{- \delta / (m_1 + m_2)} (r_i)
	\ge
	C
	(r_{i+1} - r_i)
	\left(
		R^{- \delta (n - m_2)}
		h_1
		\left(
			\frac{1}{R}
		\right)
		h_2
		\left(
			\frac{1}{R}
		\right)
	\right)^{1 / (m_1 + m_2)}
$$
for all $i \in \Xi_2$ satisfying the condition $i \le l - 2$, whence in accordance with~\eqref{PL3.7.9} it follows that
$$
	\sum_{
		i \in \Xi_2,
		\;
		i \le l - 2
	}
	J_1^{- \delta / (m_1 + m_2)} (r_i)
	\ge
	C
	\left(
		R^{m_1 + m_2 - \delta (n - m_2)}
		h_1
		\left(
			\frac{1}{R}
		\right)
		h_2
		\left(
			\frac{1}{R}
		\right)
	\right)^{1 / (m_1 + m_2)}.
$$
Combining the last estimate with the evident inequality
$$
	\sum_{
		i \in \Xi_2,
		\;
		i \le l - 2
	}
	J_1^{- \delta / (m_1 + m_2)} (r_i)
	\le
	J_1^{- \delta / (m_1 + m_2)} (R)
	\sum_{i=0}^{l-1}
	2^{- \delta i / (m_1 + m_2)}
	\le
	\frac{
		J_1^{- \delta / (m_1 + m_2)} (R)
	}{
		1 - 2^{- \delta / (m_1 + m_2)}
	},
$$
we obtain
\begin{equation}
	J_1^{- \delta} (R)
	\ge
	C
	\left(
		1 - 2^{- \delta / (m_1 + m_2)}
	\right)^{m_1 + m_2}
	R^{m_1 + m_2 - \delta (n - m_2)}
	h_1
	\left(
		\frac{1}{R}
	\right)
	h_2
	\left(
		\frac{1}{R}
	\right).
	\label{PL3.8.3}
\end{equation}
From~\eqref{PL3.7.9}, it follows that $l \ge 2$ and, therefore, $J_1 (2 R) \ge 2 J_1 (R)$.
Thus,
$$
	J_1^{- \delta} (R)
	-
	J_1^{- \delta} (2 R)
	\ge
	(1 - 2^{- \delta})
	J_1^{- \delta} (R)
$$
and~\eqref{PL3.8.3} implies~\eqref{L3.8.1}.
\end{proof}

\begin{proof}[Proof of Theorem~\ref{T2.1}]
Assume the converse, let $u_1$ and $u_2$ be a non-trivial solutions of~\eqref{1.1}.  By Lemma~\ref{L3.7}, there exists a real number $R_0$ such that $J_1 (R_0) > 0$ and, moreover,~\eqref{L3.7.1} holds for all $R \ge R_0$.

Let us put $R_i = 2^i R_0$, $i = 1,2,\ldots$. Applying Lemma~\ref{L3.7}, we have
$$
	J_1^{1 - \mu_1 \mu_2} (R_i)
	-
	J_1^{1 - \mu_1 \mu_2} (R_{i+1})
	\ge
	C
	R_i^{n + m_1 \mu_2 - (n - m_2) \mu_1 \mu_2}
	h_1^{\mu_2} 
	\left(
		\frac{1}{R_i}
	\right)
	h_2
	\left(
		\frac{1}{R_i}
	\right),
	\;
	i = 0,1,2,\ldots,
$$
whence in accordance with the inequality
\begin{align*}
&
	R_i^{n + m_1 \mu_2 - (n - m_2) \mu_1 \mu_2}
	h_1^{\mu_2} 
	\left(
		\frac{1}{R_i}
	\right)
	h_2
	\left(
		\frac{1}{R_i}
	\right)
	\\
	&
	\quad
	{}
	\ge
	C
	\int_{
		R_i
	}^{
		R_{i+1}
	}
	R^{n + m_1 \mu_2 - (n - m_2) \mu_1 \mu_2 - 1}
	h_1^{\mu_2} 
	\left(
		\frac{1}{R}
	\right)
	h_2
	\left(
		\frac{1}{R}
	\right)
	\,
	dR
\end{align*}
which is valid due to the monotonicity of the functions $h_ 1$ and $h_ 2$, it follows that
\begin{align*}
&
	J_1^{1 - \mu_1 \mu_2} (R_i)
	-
	J_1^{1 - \mu_1 \mu_2} (R_{i+1})
	\\
	&
	\quad
	{}
	\ge
	C
	\int_{
		R_i
	}^{
		R_{i+1}
	}
	R^{n + m_1 \mu_2 - (n - m_2) \mu_1 \mu_2 - 1}
	h_1^{\mu_2} 
	\left(
		\frac{1}{R}
	\right)
	h_2
	\left(
		\frac{1}{R}
	\right)
	\,
	dR,
	\quad
	i = 0,1,2,\ldots.
\end{align*}
Summing the last expression, we obtain
$$
	J_1^{1 - \mu_1 \mu_2} (R_0)
	-
	J_1^{1 - \mu_1 \mu_2} (\infty)
	\ge
	C
	\int_{
		R_0
	}^\infty
	R^{n + m_1 \mu_2 - (n - m_2) \mu_1 \mu_2 - 1}
	h_1^{\mu_2} 
	\left(
		\frac{1}{R}
	\right)
	h_2
	\left(
		\frac{1}{R}
	\right)
	\,
	dR.
$$
It is easy to see that~\eqref{T2.1.2} is equivalent to the relation
$$
	\int_{
		R_0
	}^\infty
	R^{n + m_1 \mu_2 - (n - m_2) \mu_1 \mu_2 - 1}
	h_1^{\mu_2} 
	\left(
		\frac{1}{R}
	\right)
	h_2
	\left(
		\frac{1}{R}
	\right)
	\,
	dR
	=
	\infty,
$$
while
$$
	J_1^{1 - \mu_1 \mu_2} (R_0)
	-
	J_1^{1 - \mu_1 \mu_2} (\infty)
	<
	\infty.
$$
This contradiction completes the proof.
\end{proof}

\begin{proof}[Proof of Corollary~\ref{C2.1}]
In view of Lemma~\ref{L3.3}, there exists a real number $\lambda > 1$ such that
\begin{equation}
	\zeta^{\lambda_1 \lambda_2}
	g_1 (\zeta)
	g_2^{\lambda_1} (\zeta)
	\ge
	\zeta^\lambda
	\label{PC2.2.1}
\end{equation}
for all $\zeta > 0$ in a neighborhood of infinity.
This obviously implies the validity of condition~\eqref{T2.1.1}.
To complete the proof, it remains to use Theorem~\ref{T2.1}.
\end{proof}

\begin{proof}[Proof of Theorem~\ref{T2.2}]
Putting $r_1 = R / 2$ and $r_2 = R$ in Lemma~\ref{L3.1}, we obtain
$$
	\frac{
		C
	}{
		R^{m_1}
	}
	\int_{B_R}
	|u_1|
	\,
	dx
	\ge
	\int_{
		B_{r_1}
	}
	f_1 (|u_2|)
	\,
	dx
$$
for all real numbers $R > 0$.
In view of the condition $m_1 \ge n$, the inequality
$$
	\frac{
		1
	}{
		R^n
	}
	\int_{B_R}
	|u_1|
	\,
	dx
	\ge
	\frac{
		1
	}{
		R^{m_1}
	}
	\int_{B_R}
	|u_1|
	\,
	dx
$$
holds for all real numbers $R \ge 1$. Thus, we have
$$
	\frac{
		C
	}{
		R^n
	}
	\int_{B_R}
	|u_1|
	\,
	dx
	\ge
	\int_{
		B_{r_1}
	}
	f_1 (|u_2|)
	\,
	dx
$$
for all $R \ge 1$.
By Lemma~\ref{L3.4}, the left-hand side of the last expression tend to zero as $R \to \infty$;
therefore, $u_2 = 0$ almost everywhere in $\mathbb R^n$. 

Since the triviality of one of the functions $u_i$, $i = 1, 2$ implies the triviality of the other one, this completes the proof.
\end{proof}

\begin{proof}[Proof of Corollary~\ref{C2.2}]
As mentioned above, Lemma~\ref{L3.3} and the condition $\lambda_1 \lambda_2 > 1$ imply the validity of~\eqref{PC2.2.1} for all sufficiently large $\zeta > 0$. Thus,~\eqref{T2.1.1} holds and the proof is completed by applying Theorem~\ref{T2.2}.
\end{proof}

\begin{proof}[Proof of Theorem~\ref{T2.3}]
Arguing by contradiction, we assume that~\eqref{1.1} has non-trivial solutions $u_1$ and $u_2$.
Take a real number $\delta > 0$ satisfying the inequality 
$$
	m_1 + m_2 - \delta (n - m_2) > 0.
$$ 
By Lemma~\ref{L3.8}, there exists a real number $R_0 > 0$ such that estimate~\eqref{L3.8.1} is valid for all $R \ge R_0$.
Since $u_1$ and $u_2$ are non-trivial solutions, one can also assume that $J_1 (R_0) > 0$.
As in the proof of Theorem~\ref{T2.1}, let $R_i = 2^i R_0$, $i = 1,2,\ldots$. 
We have
$$
	J_1^{- \delta} (R_i)
	-
	J_1^{- \delta} (R_{i+1})
	\ge
	C
	\delta^{m_1 + m_2 + 1}
	R_i^{m_1 + m_2 - \delta (n - m_2)}
	h_1
	\left(
		\frac{1}{R_i}
	\right)
	h_2
	\left(
		\frac{1}{R_i}
	\right),
	\quad
	i = 0,1,2,\ldots,
$$
whence it follows that
$$
	J_1^{- \delta} (R_0)
	-
	J_1^{- \delta} (\infty)
	\ge
	C
	\delta^{m_1 + m_2 + 1}
	\sum_{i=0}^\infty
	R_i^{m_1 + m_2 - \delta (n - m_2)}
	h_1
	\left(
		\frac{1}{R_i}
	\right)
	h_2
	\left(
		\frac{1}{R_i}
	\right).
$$
According to Lemma~\ref{L3.3}, the right-hand side of the last expression is equal to infinity, while the left-hand side is bounded from above by $J_1^{- \delta} (R_0)$.

This contradiction proves the theorem.
\end{proof}

\begin{proof}[Proof of Corollary~\ref{C2.3}]
By Lemma~\ref{L3.3}, the inequality $\lambda_1 \lambda_2 > 1$ implies~\eqref{T2.1.1}. Thus, to complete the proof, it remains to use Theorem~\ref{T2.3}.
\end{proof}

\begin{proof}[Proof of Corollary~\ref{C2.4}]
For $\lambda_1 = \lambda_2 = 1$, condition~\eqref{T2.1.1} takes the form~\eqref{C2.4.1}. Thus, Corollary~\ref{C2.4} follows immediately from Theorem~\ref{T2.3}.
\end{proof}

\end{document}